\documentclass[11pt]{amsart}
\usepackage{amsfonts,amssymb,amsmath,amsthm,amscd}
\usepackage{url}
\usepackage{enumerate}
\usepackage[all,arc]{xy}
\usepackage{color}
\usepackage{graphicx}
\usepackage{eucal}

\newtheorem{theorem}{Theorem}[section]
\newtheorem{lemma}[theorem]{Lemma}
\newtheorem{proposition}[theorem]{Proposition}
\newtheorem{corollary}[theorem]{Corollary}
\theoremstyle{definition}
\newtheorem{definition}[theorem]{Definition}
\newtheorem{remark}[theorem]{Remark}
\numberwithin{equation}{section}
\def\mathrmdef#1{\expandafter\def\csname#1\endcsname{{\rm#1}}}
\mathrmdef{b}\mathrmdef{coc} \mathrmdef{Dom}\mathrmdef{ev}\mathrmdef{id}\mathrmdef{coker}\mathrmdef{ker}\mathrmdef{Ker}\mathrmdef{max}
\mathrmdef{Ob}\mathrmdef{Inj}\mathrmdef{op} \mathrmdef{S}\mathrmdef{Sec}\mathrmdef{tf}\mathrmdef{tr}

 \def\mathsfdef#1{\expandafter\def\csname#1\endcsname{{\rm\sf#1}}}
 \mathsfdef{A} \mathsfdef{Ab}\mathsfdef{B}
  \mathsfdef{C}  \mathsfdef{Cat} \mathsfdef{CCoAlg}\mathsfdef{Grp}\mathsfdef{Grpd} \mathsfdef{Haus}\mathsfdef{Ind}
\mathsfdef{Mod}\mathsfdef{Set}\mathsfdef{Sets} \mathsfdef{Top}\mathsfdef{TotDis}\mathsfdef{X}\mathsfdef{Y}\mathsfdef{XMod} \mathsfdef{Vect}

\def\bkF{\mathbb{F}}

\def\bkZ{\mathbb{Z}}
\def\bkQ{\mathbb{Q}}

\author[M.M.~Clementino]{Maria Manuel Clementino}
\address{CMUC, Department of Mathematics, University of Coimbra, 3001-501 Coimbra, Portugal}\thanks{Partially supported by  Centro de Matem\'{a}tica da Universidade de
  Coimbra -- UID/MAT/00324/2013, funded by the Portuguese Government
  through FCT/MCTES and co-funded by the European Regional Development Fund
  through the Partnership Agreement PT2020, and by the FCT Sabbatical Grant SFRH/BSAB/127925/2016.}
\email{mmc@mat.uc.pt}

\author[M.~Gran]{Marino Gran}
\address{Institut de Recherche en Math\'{e}matique et Physique, Universit\'{e} catholique de Louvain, Chemin du Cyclotron 2, 1348 Louvain-la-Neuve, Belgium}
\email{marino.gran@uclouvain.be}

\author[G.~Janelidze]{George Janelidze}
\address{Department of Mathematics and Applied Mathematics, University of Cape Town, Rondebosch 7701, Cape Town, South Africa}
\thanks{Partially supported by South African NRF}
\email{George.Janelidze@uct.ac.za}


\begin{document}

\title[Some remarks on protolocalizations and protoadditive reflections]{Some remarks on protolocalizations and protoadditive reflections}

\begin{abstract}
We investigate additional properties of protolocalizations, introduced and studied by F. Borceux,
M. M. Clementino, M. Gran, and L. Sousa, and of protoadditive reflections, introduced and studied by T. Everaert and M. Gran. Among other things we show that there are no non-trivial (protolocalizations and) protoadditive reflections of the category of groups, and establish a connection between protolocalizations and Kurosh--Amitsur radicals of groups with multiple operators whose semisimple classes form subvarieties.

\end{abstract}
\maketitle

\section{Introduction}
Consider a reflection
\begin{equation}
\label{eq:refl}
\xymatrix@!C=48pt{ \X  \ar@<-1ex>[r]\ar@{}[r]|{\perp} & \C
\ar@<-1ex>[l]_F  }
\end{equation}
 in which $\C$ is a homological category (in the sense of \cite{BB2004}), $\X$ is a full replete reflective subcategory of $\C$, and $F$ is the left adjoint of the inclusion functor $\X \to \C$; it will be often
convenient to say that $F$ itself is the reflection $\C\ \to \X$.
We will be interested in the following conditions on the functor $F$:
\begin{enumerate}[(a)]
\item $F$ preserves all finite limits; that is, $F$ is a \emph{localization};
\item $F$ preserves kernels of regular epimorphisms, or, equivalently, short exact sequences, in which case it is called a \emph{protolocalization} \cite{BCGS08};
\item $F$ preserves kernels of split epimorphisms, or, equivalently, split short exact sequences, in which case it is called a \emph{protoadditive reflection} \cite{EG1} (see also \cite{EG2});
\item $F$ preserves finite products.
\end{enumerate}
If \C\ is an abelian category, (a) is equivalent to (b) and (c) is equivalent to (d), which is surely the reason why conditions (b) and (c) were only introduced recently. In this paper we complement the properties and examples of protolocalizations and protoadditive reflections presented in \cite{BCGS08, EG1, EG2}. Among our contributions we highlight the following:
%If \C\ is an abelian category, (a) is equivalent to (b) and (c) is equivalent to (d), which is surely the reason why conditions (b) and (c) were only introduced recently. We know many properties related to them and many examples of functors satisfying them, and yet, this is only the very beginning of a new theory, as the readers will surely conclude from the new remarks the present paper is devoted to. These remarks are:
\begin{enumerate}[1.]
\item	 In a regular category, the factors of the (regular epi, mono)-factorization of a protoadditive reflection are also protoadditive (Theorem \ref{protoadditive}); in general this is not valid for protolocalizations.
\item	There is no non-trivial protoadditive reflection of the category of groups (and therefore the same is true for protolocalizations) (Theorem \ref{thm:no_proto}); there are, however, several known examples of protoadditive reflections and protolocalizations of the category of internal groups in a category (Remark \ref{remarks}).
\item	Being a (regular epi)-reflective protolocalization of a variety of groups with multiple operators turns out to be closely related to other conditions considered in several areas of algebra. In particular every protolocalization of such a variety is
\[\mbox{semisimple = attainable = admissible = semi-left-exact = fibered}\]
(see Theorem \ref{thrm:AA.6}, Proposition \ref{thrm:AA.8}, and Remark \ref{rem:AA.9} for details).
\item There are no non-trivial localizations, neither of the category of rings nor of the category of commutative rings, that are reflections onto subvarieties (Remark \ref{AA.10}.\eqref{AA.10.1}).
\item The reflection of the variety of commutative von Neumann regular rings to the variety of Boolean rings is not a localization (Remark \ref{AA.10}.\eqref{AA.10.3}). This negatively answers a question asked in \cite{BCGS08}.
\end{enumerate}
{\bf Acknowledgement.} The authors are very grateful to Francesca Cagliari for many discussions and interesting new observations of possible reflections of the category of topological groups, including pointing out the existence of the localization mentioned as Remark 3.2(b) in this paper.

\section{Protolocalizations and protoadditive reflections}\label{sect:2}
In this section we shall always assume that $\C$ is a \emph{homological} \cite{BB2004, BC05} category. Recall that
 a finitely complete pointed category \C\ is homological when it is also \begin{itemize}
\item  \emph{regular}: any arrow in $\C$ factors as a regular epimorphism followed by a monomorphism, and these factorizations are pullback-stable;
\item and \emph{protomodular} \cite{Bourn0}: given any commutative diagram in $\C$
 \begin{equation}\label{eq:protoloc}
\xymatrix@!C=48pt{0\ar[r]&K\ar[r]^k\ar[d]_{u}&A \ar@<.5ex>[r]^f\ar[d]^{v}&B \ar@<.5ex>[l]^s \ar[d]^{w}\ar[r]&0 \\
0\ar[r]&K' \ar[r]_{k'}& A'  \ar@<.5ex>[r]^{f'} &B'  \ar@<.5ex>[l]^{s'}\ar[r]&0 }
\end{equation}
where $k = \ker (f)$, $k' = \ker(f')$, $f \cdot s = 1_B$, $f' \cdot s' = 1_B'$, then $v$ is an isomorphism whenever both $u$ and $w$ are isomorphisms.
\end{itemize}
 There are many examples of homological categories, such as the categories of groups, Lie and other kinds of non-unital algebras over rings,
crossed modules, and Heyting semi-lattices. For a pointed (finitary) variety C of universal algebras the following conditions are actually equivalent:
\begin{itemize}
\item $\C$ is homological;
\item $\C$ is semi-abelian \cite{JMT};
\item $\C$ is protomodular \cite{Bourn0};
\item $\C$ is classically ideal determined in the sense of \cite{U}, which is the same as BIT speciale in the sense of \cite{U1}.
\end{itemize}
There are other examples of homological categories that are not (finitary) varieties of universal algebras, for instance ${\mathbb C}^*$-algebras, topological groups, Hausdorff groups and, more generally, any model of a semi-abelian theory in the category of topological (or Hausdorff) spaces \cite{BC05}. The dual of the category of pointed sets, as also the dual of pointed objects of any topos, are homological categories \cite{BB2004}.

Let  $\X$ be a reflective (full and replete) subcategory of $\C$, with reflection $F \colon \C \to \X$,
and unit $\eta=(\eta_C:C\to F(C))_{C\in\C}$. %(Here by subcategory we will always mean full and replete subcategory.)
As observed in \cite{BCGS08}, $\X$ is also a homological category provided that it is regular. The following two notions were introduced in \cite{BCGS08} and \cite{EG1}, respectively.

\begin{definition}
\begin{enumerate}
\item The reflection \eqref{eq:refl} is said to be a \emph{protolocalization} if $\X$ is regular and $F$ preserves short exact sequences, that is, in the diagram
\begin{equation}\label{eq:protoloc}
\xymatrix@!C=48pt{0\ar[r]&K\ar[r]^k\ar[d]_{\eta_K}&A\ar[r]^f\ar[d]^{\eta_A}&B\ar[d]^{\eta_B}\ar[r]&0\\
0\ar[r]&F(K)\ar[r]^{F(k)}&F(A)\ar[r]^{F(f)}&F(B)\ar[r]&0}
\end{equation}
if the top line is a short exact sequence in $\C$, so that $f=\coker(k)$ and $k=\ker(f)$, then the bottom line is a short exact sequence in $\X$. We will also say that $F : \C\to\X$ is a protolocalization of \C.
\item The reflection \eqref{eq:refl} is said to be \emph{protoadditive} if $F$ preserves split short exact sequences, that is, in the diagram
\begin{equation}\label{eq:protoadd}
\xymatrix@!C=48pt{0\ar[r]&K\ar[r]^k\ar[d]_{\eta_K}&A\ar@<.5ex>[r]^f\ar[d]^{\eta_A}&B\ar@<.5ex>[l]^s\ar[d]^{\eta_B}\ar[r]&0\\
0\ar[r]&F(K)\ar[r]^{F(k)}&F(A)\ar@<.5ex>[r]^{F(f)}&F(B)\ar@<.5ex>[l]^{F(s)}\ar[r]&0}
\end{equation}
with $f\cdot s=1_B$, if the top line is a short exact sequence in $\C$, then the bottom line is a short exact sequence in $\X$. We will also say that $F : \C\to\X$ is a protoadditive reflection of \C.
\end{enumerate}
\end{definition}

\begin{remark}
\begin{enumerate}
\item As shown in \cite{BCGS08}, given a reflection \eqref{eq:refl}, $\X$ is regular if, and only if,
$F$ preserves pullbacks of the form
\[\xymatrix{B \times_Y X \ar[r]\ar[d]&X\ar[d]^f\\
B\ar[r]_b&Y}\]
where $f\in\X$, and $b$ is the (regular) image in the regular category $\C$ of a morphism in $\X$. In particular, when $\X$ is \emph{(regular epi)-reflective}, that is, if $\eta_C$ is a regular epimorphism for every object $C$ of \C, then $F$ preserves these pullbacks, and so \X\ is regular.
\item Every protolocalization of a homological category is homological. Indeed, if \C\ is pointed, or \C\ is protomodular, then so is any reflective subcategory of \C. Only regularity does not follow for free.
\end{enumerate}
\end{remark}

\begin{lemma}\label{thm:1.4}
\begin{enumerate}[\em (1)]
\item For a (regular epi)-reflection \eqref{eq:refl} and a short exact sequence%\newline $\xymatrix{0\ar[r]&K\ar[r]^k&A\ar[r]^f&B\ar[r]&0}$
\begin{equation}\label{eq:ses}
\xymatrix@!C=48pt{0\ar[r]&K\ar[r]^k&A\ar[r]^f&B\ar[r]&0}
\end{equation}
in $\C$, the following conditions are equivalent, for diagram \eqref{eq:protoloc}:
\begin{enumerate}[\em (i)]
\item the bottom line is a short exact sequence in $\X$;
\item the bottom line is a short exact sequence in $\C$.
\end{enumerate}
\item For a reflection \eqref{eq:refl} and a split short exact sequence
\begin{equation}\label{eq:sses}
\xymatrix@!C=48pt{0\ar[r]&K\ar[r]^k&A\ar@<.5ex>[r]^f&B\ar@<.5ex>[l]^s\ar[r]&0}
\end{equation}
in $\C$, the following conditions are equivalent, for diagram \eqref{eq:protoadd}:
\begin{enumerate}[\em (i)]
\item the bottom line is a short exact sequence in $\X$;
\item the bottom line is a short exact sequence in $\C$.\hfill$\Box$
\end{enumerate}
\end{enumerate}
\end{lemma}

We recall that a subcategory $\X$ is said to have the \emph{2-out-of-3 property} if, for any short exact sequence \eqref{eq:ses}
in $\C$, if any two of the three objects $K,A,B$ belong to $\X$, then the third one also belongs to $\X$. $\X$ is said to be \emph{extension closed} if, for any short exact sequence \eqref{eq:ses}, with $K$ and $B$ in $\X$ also $A$ belongs to $\X$. Analogously, one can define the \emph{split 2-out-of-3-property} and \emph{split extension closed} subcategory.

\begin{proposition}
\begin{enumerate}[\em (1)]
\item If $\X$ is a (regular epi)-reflective protolocalization of $\C$, then $\X$ has the 2-out-of-3 property. In particular, $\X$ is extension closed.
\item If \X\ is a protoadditive reflection, so in particular when it is a protolocalization, then \X\ has the split 2-out-of-3 property. Consequently, $\X$ is split extension closed.
\end{enumerate}

\end{proposition}
\begin{proof}
(1): %First of all we point out that the 2-out-of-3 property subsumes closure under extensions.
Given a short exact sequence \eqref{eq:ses}, if $f:A\to B$ belongs to $\X$, then its kernel also belongs to $\X$. If $K$ and $A$ belong to $\X$, then we apply $F$ to the sequence \eqref{eq:ses} and conclude, by Lemma \ref{thm:1.4}, that $B\in\X$. If $K$ and $B$ belong to $\X$, then we apply $F$ again and use the Short Five Lemma.

(2) can be shown analogously.
\end{proof}

\begin{remark}\label{rem:GJ}
The hypothesis that the reflections are regular epimorphisms in statement (1) of Lemma \ref{thm:1.4} is essential.
Indeed, we take \C\ to be the category of pairs $(A,R)$, where $A$ is an abelian group and $R$ is a subgroup of $A\times A$, and $f:(A,R)\to(B,S)$ is a morphism in \C\ if $f:A\to B$ is a homomorphism with $(f\times f)(R)\subseteq S$. This category is additive and homological, since it is a full reflective subcategory, closed under subobjects, of the abelian category of internal graphs of the category of abelian groups. Let \X\ be the subcategory of \C\ with objects all $(A,R)$ in which $R$ is a transitive relation on $A$.
The reflection $F : \C\to\X$ has $F(A,R)=(A,R^\tr)$, where $R^\tr$ is the smallest transitive relation on $A$ which is a subgroup of $A\times A$ containing $R$. This is a monoreflection of course.
Take $A = \bkZ^4$ (where $\bkZ$ is the additive group of integers, although we could take any non-trivial abelian group instead of it) and
$R =\{((m,0,n,0),(0,m,0,n)) \in A\times A\, |\, m, n \in\bkZ\}$.
Take $B = \bkZ^3$, define $f : A\to B$ by $f(m,n,p,q)=(m,n+ p,q)$, and take $S=(f\times f)(R)$; $f:(A,R) \to (B,S)$ is a regular epimorphism in \C, and, together with its kernel, it gives a short exact sequence in \C:
\begin{equation}\label{eq:example}
\xymatrix{0\ar[r]&(\Ker(f),\{((0,0,0,0),(0,0,0,0))\})\ar[r]&(A,R)\ar[r]^f&(B,S)\ar[r]&0.}
\end{equation}
Both $(A,R)$ and the kernel of $f:(A,R)\to(B,S)$ belong to \X, but $(B,S)$ does not, since $S$ is not transitive. Therefore the image of \eqref{eq:example} under $F$ is a short exact sequence in \X\ -- since $F$ preserves cokernels -- but it is not a short exact sequence in \C.
\end{remark}

\begin{lemma}\cite[Proposition 2.5]{EG2}\label{mono}
If $F$ is a (regular epi)-reflection then it is protoadditive if, and only if, for any split short exact sequence \eqref{eq:sses}, $F(k)$ is a monomorphism.
\end{lemma}

\begin{remark}
The corresponding result for protolocalizations is false, as the example of Remark \ref{rem:protoloc} shows.
\end{remark}

\begin{proposition}
If \C\ is a regular category, any reflection \eqref{eq:refl} factors as a (regular epi)-reflection $F' \colon \C \rightarrow \Y$ followed by a monoreflection $F'' \colon \Y \rightarrow \X$, as in the diagram:
\begin{equation}\label{eq:fact}
\xymatrix@!C=30pt{ {\;\;\;\X\;\;\;}  \ar@<-1ex>[rr]\ar@{}[rr]|{\perp}\ar@<-.5ex>[rdd] && {\;\;\;\C\;\;\;}\ar@<-.5ex>[ldd]_{F'}\ar@<-1ex>[ll]_{F}\\
\\
&{\;\;\Y\;\;} \ar@<-1.5ex>[uur]\ar@<-1.5ex>[luu]_{F''}\ar@{}[uur]|{\perp}\ar@{}@<-.5ex>[uul]|{\perp} }
\end{equation}
where $\Y=\{Y\in\C\,|\, \mbox{ there is a monomorphism }Y\to X\mbox{ with }X\in\X\}$.
Moreover, as any monoreflection is an epi-reflection, $F''$ is a bireflection, that is, its unit is pointwise both an epimorphism and a monomorphism.
\end{proposition}
\begin{proof}
Let $\eta$, $\eta'$ and $\eta''$ be the units of the adjunctions for the reflections $F$, $F'$ and $F''$, respectively. For each object $C$ of \C, in the commutative diagram
\begin{equation}\label{eq:factor}
\xymatrix{C\ar[rr]^{\eta_C}\ar[dr]_{\eta'_C}&&F(C)\\
&F'(C)\ar[ru]_{\eta''_{F'(C)}}}
\end{equation}
we have:
\begin{enumerate}[--]
\item $\eta'_C : C \rightarrow F'(C)$ is a universal arrow giving the reflection of $C$ into $\Y$;
\item $\eta''_{F'(C)}:F'(C)\to F(C)$ is a universal arrow giving the reflection of $F'(C)$ into $\X$.
\end{enumerate}
\end{proof}
\begin{remark}
Observe that this result follows from a more general one that holds whenever there is a proper factorization system $(\mathcal E, \mathcal M)$ in $\C$. Under those assumptions the statement of the theorem would then give a canonical factorization of the reflection $F :  \C \rightarrow X$ as an $\mathcal E$-reflection followed by an $\mathcal M$-reflection. About the study of such factorization see also \cite{Baron}.
\end{remark}
\begin{theorem}\label{protoadditive}
Let \C\ be a regular category. In diagram \eqref{eq:fact} $F$ is a protoadditive reflection if, and only if, $F'$ and $F''$ are protoadditive.
\end{theorem}
\begin{proof}
If $F'$ and $F''$ are protoadditive, then $F=F'\cdot F''$ is clearly protoadditive.

Now assume that $F$ is protoadditive, and consider the diagram
\begin{equation}
\label{eq:fact2}
\xymatrix@!C=48pt{0\ar[r]&K\ar[r]^k\ar[d]_{\eta'_K}&A\ar[d]^{\eta'_A}\ar@<.5ex>[r]^f&B\ar@<.5ex>[l]^s\ar[r]\ar[d]^{\eta'_B}&0\\
0\ar[r]&F'(K)\ar[r]^{F'(k)}\ar[d]_{\eta''_{F'(K)}}&F'(A)\ar[d]^{\eta''_{F'(A)}}\ar@<.5ex>[r]^{F'(f)}&F'(B)\ar@<.5ex>[l]^{F'(s)}
\ar[r]\ar[d]^{\eta''_{F'(B)}}&0\\
0\ar[r]&F(K)\ar[r]^{F(k)}&F(A)\ar@<.5ex>[r]^{F(f)}&F(B)\ar@<.5ex>[l]^{F(s)}\ar[r]&0}
\end{equation}
where the top row is an exact sequence, $\eta'=(\eta'_C)_{C\in\C}$ and $\eta''=(\eta''_Y)_{Y\in\Y}$ are the units of the adjunctions, so that $\eta_C=\eta''_{F'(C)}\cdot\eta'_C$ for every object $C$ of \C\  (as in \eqref{eq:factor}). By assumption $F(k)$ is a monomorphism, as well as $\eta''_{F'(K)}$, and so $F'(k)$ has to be a monomorphism. With Lemma \ref{mono} we conclude that $F'$ is protoadditive. To show that $F''$ is also protoadditive, consider diagram \eqref{eq:fact2} with $K,A,B$ in \Y. In this case the two top rows are isomorphic and the conclusion follows.
\end{proof}

\begin{remark}\label{rem:protoloc}
The previous statement is not valid when we replace protoadditive reflection with protolocalization. Consider the reflection
\[\xymatrix@!C=48pt{\Vect_\bkQ\ar@<-1ex>[r]\ar@{}[r]|-{\perp}& \Ab\ar@<-1ex>[l]_-{\bkQ\otimes -}}\]
where $\bkQ$ is the field of rational numbers, $\Vect_\bkQ$ is the category of vector spaces over $\bkQ$, \Ab\ is the category of abelian groups, and $\otimes$ is the tensor product. This reflection is a localization because $\bkQ$, being a flat module over the ring $\bkZ$ of integers, makes $\bkQ\otimes-$ an exact functor (between abelian categories).

Our factorization of this reflection is
\[\xymatrix@!C=48pt{\Vect_\bkQ\ar@<-1ex>[r]\ar@{}[r]|-{\perp}&\Ab_\tf\ar@<-1ex>[l]_-{F''}\ar@<-1ex>[r]\ar@{}[r]|-{\perp}&\Ab\ar@<-1ex>[l]_-{F'}}\]
where $\Ab_\tf$ is the category of torsion-free abelian groups.
But the reflector $F'$ is not a protolocalization. Indeed, the $F'$-image of the short exact sequence
\[\xymatrix@!C=48pt{0\ar[r]&\bkZ\ar[r]^{-\times n}&\bkZ\ar[r]&\bkZ/n\bkZ\ar[r]&0,}\]
where $n$ is any natural number larger than 1, is
\[\xymatrix@!C=48pt{0\ar[r]&\bkZ\ar[r]^{-\times n}&\bkZ\ar[r]&0\ar[r]&0,}\]
which is not a short exact sequence in the category of torsion-free abelian groups.
\end{remark}

\section{In $\Grp$ protolocalizations and protoadditive reflections trivialize}

\begin{theorem}\label{thm:no_proto}
There are no non-trivial protoadditive reflections of the category \Grp\ of groups. That is, if $F : \Grp\to\X$ is a protoadditive reflection, then either $\X = \Grp$ or \X\ consists of trivial groups.
\end{theorem}

\begin{proof}
Let $\X$ be a protoadditive subcategory of $\Grp$, with reflection $F:\Grp\to\X$ and unit $\eta$ as in the previous section.
Let $X\in\X$ and let $B$ be any group. Form the kernel $\xymatrix{B\flat X\ar[r]^-k&B+X}$ of the split epimorphism $\xymatrix{B+X\ar[r]^-{[1,0]}&B}$, and consider the diagram
\[\xymatrix@!C=48pt{&0\ar[d]&0\ar[d]&0\ar[d]&\\
0\ar[r]&\Ker(\eta_{B\flat X})\ar[r]\ar[d]^{\ker(\eta_{B\flat X})}&\Ker(\eta_{B+X})\ar[r]\ar[d]^{\ker(\eta_{B+X})}&\Ker(\eta_B)\ar[r]\ar[d]^{\ker(\eta_B)}&0\\
0\ar[r]&B\flat X\ar[r]^k\ar[d]^{\eta_{B\flat X}}&B+X\ar[r]^{[1,0]}\ar[d]^{\eta_{B+X}}&B\ar[r]\ar[d]^{\eta_B}&0\\
0\ar[r]&F(B\flat X)\ar[r]^{F(k)}\ar[d]&F(B+X)\ar[r]^{F([1,0])}\ar[d]&F(B)\ar[r]\ar[d]&0\\
&0&0&0&}\]
Let $a\in\Ker(\eta_B)$. Then, for any $x\in X$, the element $(x,a,x^{-1},a^{-1})$ of $B+X$ is both in $\Ker(\eta_{B+X})$ and $B\flat X$. Since $F(k)$ is a monomorphism, we conclude that $\eta_{B\flat X}(x,a,x^{-1},a^{-1})=1$.

Now, since $B\flat X$ is the coproduct of $B$ copies of $X$, with coproduct injections $\iota_b:X\to B\flat X$ given by $\iota_b(x)=(b,x,b^{-1})$ for every $b\in B$ and $x\in X$, and $F$ preserves coproducts, $F(\iota_b):X\to F(B\flat X)$ are coproduct injections in $\X$. Hence, for $a\in\Ker(\eta_B)$ as above,
\[\begin{array}{rcl}
1&=&\eta_{B\flat X}(x,a,x^{-1},a^{-1})=\eta_{B\flat X}(x)\eta_{B\flat X}(a,x^{-1},a^{-1})\vspace*{1mm}\\
&=&\eta_{B\flat X}(\iota_1(x))\eta_{B\flat X}(\iota_a(x^{-1}))=F(\iota_1)(x)F(\iota_a)(x^{-1}),\end{array}\]
and then we can conclude that the coproduct injections $F(\iota_1)$ and $F(\iota_a)$ are equal. Since $\X$ is pointed, this implies $a=1$, and so $\eta_B$ is a monomorphism for every group $B$. Since every monoreflection is an epi-reflection, $F$ is a bireflection, which means that $F$ is an isomorphism since the category of groups is balanced.
\end{proof}

\begin{remark}\label{remarks}
Let \A\ be a category with finite limits and $\Grp(\A)$ the category of internal groups in \A. In spite of Theorem \ref{thm:no_proto}, there are many non-trivial protolocalizations and protoadditive reflections of the form $\Grp(\A)\to\X$, with homological $\Grp(\A)$, including the following ones:
\begin{enumerate}[(a)]
\item All localizations $\A\to\X$, where \A\ is additive homological, since \A\ can be identified with $\Grp(\A)$ whenever \A\ is additive.

\item The reflection $\pi_0 : \Grp(\Cat)\to\Grp$, where \Cat\ is the category of (small) categories and $\pi_0$ sends internal groups in \Cat\ to the groups of isomorphism classes of their objects (and \Grp\ is identified with the category of internal groups in the category of discrete categories), is protoadditive but not a protolocalization. Moreover, replacing $\Grp(\Cat)$ with its isomorphic category $\Cat(\Grp)$, one can then prove protoadditivity for the more general reflection $\pi_0: \Cat(\C)\to\C$, where \C\ is any semi-abelian category in the sense of \cite{JMT} (recall that a homological category is semi-abelian when it is also Barr-exact and finitely cocomplete). This is done in \cite{EG1}, and, as shown in \cite{BCGS08}, this reflection is a protolocalization when \C\ is also arithmetical in the sense of \cite{Pe}. Let us also recall that $\Grp(\Cat) = \Grp(\Grpd)$, where \Grpd\ is the category of groupoids, can be identified (up to equivalence of categories) with the category $\mathsf{X}\Mod$ of crossed modules, and then $\pi_0$ will become the familiar cokernel functor to \Grp. The same can be done with any semi-abelian \C\ using internal crossed modules in the sense of \cite{J}.

\item The forgetful functor $F$ from the category $\Grp(\Top)$ of topological groups to \Grp, if \Grp\ is identified with the category of indiscrete topological groups; since $F$ has both a left and a right adjoint, it is a localization.

\item As shown in \cite{EG2}, the reflection $\Grp(\Top)\to \Grp(\Haus)$, considered in \cite{BC05}, where $\Grp(\Haus)$ is the category of Hausdorff topological groups, is a protoadditive reflection. However, it is obviously not a protolocalization: for, just apply it to, say, $0\to{\mathbb Q}\to{\mathbb R}\to{\mathbb R}/{\mathbb Q}\to 0$, where ${\mathbb Q}$ and ${\mathbb R}$ are the additive groups of rational and real numbers, respectively.

\item The reflection $\Grp(\Haus)\to\Grp(\TotDis)$, where $\Grp(\TotDis)$ is the category of totally disconnected topological groups, is another protoadditive reflection that is not a protolocalization; this time
    $0\to{\mathbb Z}\to{\mathbb R}\to{\mathbb R}/{\mathbb Z}\to 0$ provides the desired counter-example. The fact that the reflection $\Grp(\Haus)\to\Grp(\TotDis)$ is protoadditive can be shown in the same way as it is done in \cite{EG2} for the reflection from the category of compact Hausdorff semi-abelian algebras to the subcategory of compact totally disconnected semi-abelian algebras (using its Theorem 2.6).

\item Let $K$ be a (unital) commutative ring and $G$ a group. The group algebra $K[G]$ has a natural Hopf $K$-algebra structure whose comultiplication $\Delta: K[G]\to K[G]\otimes_K K[G]$ is defined by
$\Delta(g) = g\otimes g$, for all $g$ in $G$. Accordingly, for any Hopf $K$-algebra $H$, an element $h$ in $H$ is said to be a group-like element if $\Delta(h) = h\otimes h$ for the comultiplication $\Delta$ of $H$. Identifying the category of cocommutative Hopf $K$-algebras with the category $\Grp(\CCoAlg_K)$ of internal groups in the category of cocommutative $K$-coalgebras (see \cite{Ross} for instance), and identifying the category of Hopf $K$-algebras of the form $K[G]$ (for all groups $G$) with the category of groups, yields now the group-like element functor $\Grp(\CCoAlg_K) \to \Grp$. As shown in \cite{GKV} (see also \cite{GKV2}), when $K$ is a field of characteristic zero, $\Grp(\CCoAlg_K)$ is semi-abelian and this functor is a localization.
\end{enumerate}
\end{remark}

\section{Admissibility and attainability}

Let $\C$ be a category with pullbacks, $\X$ a full replete subcategory of $\C$, and $F :\C \to \X$ the reflection. Following \cite{J1984}, we shall call an object $A$ in $\C$ \emph{$F$-admissible} if the right adjoint of the induced functor
\[\xymatrix@!C=90pt{F^A : (\C\downarrow A) \ar[r]& (\X\downarrow F(A))}\]
is fully faithful, or, equivalently, for every morphism $f: X \to F(A)$ in $\X$, the canonical morphism
\begin{equation}
\label{AA.1}
\xymatrix@!C=90pt{F(A\times_{F(A)}X) \ar[r]& F(X) = X}
\end{equation}
is an isomorphism. The canonical morphism $A\to F(A)$ involved in the pullback above will be denoted by $\eta_A$, and, when $\C$ is pointed, the kernel of $\eta_A$ will be denoted by $\kappa_A : K(A)\to A$. From now on we will be working in a \emph{normal category} \C\ (in the sense of Z. Janelidze \cite{J2010}); that is, \C\ is a pointed regular category where every regular epimorphism is a normal epimorphism. The following theorem generalizes Theorem 3.1 of \cite{J1989}:

\begin{theorem}
\label{thrm:AA.1}
If $\C$ is a normal category and $F : \C \to \X$ is a (regular epi)-reflection (=(normal epi)-reflection), then the following conditions on an object $A$ in $\C$ are equivalent:
\begin{enumerate}[\em (i)]
\item $A$ is $F$-admissible;
\item $F(K(A)) = 0$;
\item $K(K(A)) \cong K(A)$ canonically.
\end{enumerate}
Moreover, if these conditions hold, and $A$ admits a short exact sequence $\xymatrix@!C=5pt{0\ar[r]&X\ar[r]&A\ar[r]&Y\ar[r]&0}$ in $\C$ with $X$ and $Y$ in $\X$, then $A$ is in $\X$.
\end{theorem}

\begin{proof}(i) $\Rightarrow$ (ii): Putting $X = 0$ in \eqref{AA.1} we obtain $F(K(A)) = F(A\times_{F(A)}0) = 0$ by (i).

(ii) $\Leftrightarrow$ (iii) is trivial.

(ii) $\Rightarrow$ (i): For a morphism $f : X \to F(A)$ in $\X$, we need to prove that (ii) implies that the canonical morphism \eqref{AA.1} is an isomorphism. For that we observe that since $\eta_A$ is a normal epimorphism, so is the pullback projection $p : A\times_{F(A)}X \to X$, and, since these two morphisms have isomorphic kernels, this gives a cokernel diagram
\begin{equation}
\label{AA.2}
\xymatrix@!C=48pt{K(A)\ar[r]& A\times_{F(A)}X\ar[r]^-p& X \ar[r] & 0.}
\end{equation}
Since $F$ is a left adjoint, it preserves cokernel diagrams, and so (ii) implies that $F(p)$ is an isomorphism, as desired.

To prove the last assertion of the theorem, we observe:
\begin{enumerate}
\item	Since $F$ is an epi-reflection, $\X$ is closed under normal subobjects in $\C$.
\item	The existence of a short exact sequence $\xymatrix@!C=5pt{0\ar[r]&X\ar[r]&A\ar[r]&Y\ar[r]&0}$ implies that $\kappa_A : K(A)\to A$ factors through $X\to A$.
\item	(1) and (2) together imply that $K(A)$ is in $\X$.
\item	(3) and (ii) together imply $K(A) = 0$.
\item	Since $F$ is a (regular epi)-reflection, (4) implies that $A$ is in $\X$.
\end{enumerate}
\end{proof}

\begin{remark}\label{rem:AA.2}
\begin{enumerate}
\item As observed in \cite{CJKP1997} any reflection $F : \C\to\X$ to a full subcategory has all objects $A$ in $\C$ admissible if and only if it is semi-left-exact in the sense of \cite{CHK}. On the other hand, for any functor $F : \C\to\X$, having a fully faithful right adjoint for $F^A : (\C\downarrow A)\to(\X\downarrow F(A))$ for every $A$ in $\C$ is equivalent to being (essentially) a fibration (cf. Proposition 36 of \cite{BCGS08}).

\item As easily follows from Proposition 19 of \cite{BCGS08}, the Galois theory of $F$ (see e.g. \cite{J1984, J1989}) becomes \emph{trivial} when $F$ is a protolocalization -- in the sense that its Galois groupoids become effective equivalence relations, and all coverings are trivial. This is not the case, however, for protoadditive reflections, for which it is often possible to provide an explicit and simple description of the (non-trivial) Galois groups in terms of Hopf formulae \cite{EG2}.
 \end{enumerate}
 \end{remark}
\begin{theorem}
 \label{thrm:AA.3}
If $\C$ is a normal category, $F : \C\to\X$ is a (regular epi)-reflection, and $\X$ is extension closed in $\C$, then the equivalent conditions {\em (i)}--{\em (iii)} of Theorem \ref{thrm:AA.1} on $A$ hold whenever $K(K(A))$ is (canonically) a normal subobject of $A$ (that is, whenever the composite of
$K(K(A))\to K(A)\to A$ is a normal monomorphism).
\end{theorem}
\begin{proof}
Take any $A$ in $\C$ and consider the commutative diagram
\begin{equation}
\label{AA.3}
\xymatrix@!C=48pt{
& 0 \ar[d]  \\
&K(K(A))\ar[dr] \ar[d] & & &\\
0\ar[r]&K(A)\ar[r] \ar[d] &A\ar[r]^{\eta_A} \ar[d] &\frac{A}{K(A)} \ar[r] \ar@{=}[d] &0 \\
& \frac{K(A)}{K(K(A))} \ar[d]  \ar[r] & \frac{A}{K(K(A))} \ar[r] & \frac{A}{K(A)} \\
& 0
}
\end{equation}
whose middle row is a short exact sequence, and the quotients in the lower row are the obvious ones. The double quotient isomorphism theorem for normal categories (see Lemma $1.3$ in \cite{monotone})
implies that the lower row is also exact.
Since $\frac{K(A)}{K(K(A))} \cong F(K(A))$ and $\frac{A}{K(A)} = F(A)$, the extension closedness of $\X$ implies that $\frac{A}{K(K(A))}$ is in $\X$. By the universal property of $\eta_A: A\to F(A)$ this implies that
$K(K(A))\cong K(A)$ canonically.
\end{proof}

From this theorem and the last assertion of Theorem \ref{thrm:AA.1}, we obtain the following result (see also the Corollary in \cite{JT}):

\begin{corollary}
\label{cor:AA.4}
Suppose $F : \C\to  \X$ is as in Theorem \ref{thrm:AA.1} and $\C$ is a normal category. Then the following conditions are equivalent:
\begin{enumerate}[\em (i)]
\item every object $A$ of $\C$ satisfies the equivalent conditions of Theorem \ref{thrm:AA.1}, or, equivalently, $F : \C\to\X$ is semi-left-exact;
\item $\X$ is extension closed in $\C$ and, for every $A$ in $\C$, $K(K(A))$ is a normal subobject of $A$.
\end{enumerate}
\end{corollary}

The assumptions of Theorems \ref{thrm:AA.1} and \ref{thrm:AA.3} and of Corollary \ref{cor:AA.4} hold whenever $\C$ is a semi-abelian category and $\X$ is a Birkhoff subcategory of $\C$, that is a full (regular epi)-reflective subcategory of $\C$ that is stable regular quotients. In particular, they hold when $\C$ is a variety of \emph{$\Omega$-groups} (=groups with multiple operators in the sense of P. J. Higgins \cite{H1956}), and $\X$ a subvariety of $\C$.

In order to state the next proposition let us recall that a set $\Phi$ of identities is \emph{attainable}  - in the sense of T. Tamura \cite{T1966} - on an $\Omega$-group $A$ if $A(A(\Phi)) = A(\Phi)$, where $$A(\Phi) =\cap  \{ I \lhd A \, \mid \, A/I  {\, \rm satisfies \, } \varphi, \forall \varphi \in \Phi \}$$
and $I \lhd A$ means that $I$ is an ideal of $A$.
We then have the following theorem, which in fact follows from previously known results, including Theorem 3.1 of \cite{J1989}:

\begin{proposition}
\label{thrm:AA.5} If $\C$ is a variety of $\Omega$-groups and $\X$ a subvariety of $\C$ determined by a set $\Phi$ of identities, then the equivalent conditions of Theorem \ref{thrm:AA.1} on an object $A$ in $\C$ are also equivalent to each of the following conditions:
\begin{enumerate}
\item[\em (iii)] $\Phi$ is attainable on $A$ in the sense T. Tamura ;
\item[\em (iv)] $A$ is $\X$-attainable in the sense of A. Mal'tsev \cite{M1967}.
\end{enumerate}
\end{proposition}

\begin{proof}
 (iii) $\Leftrightarrow$ (iv) is well-known in the more general context of arbitrary varieties (in fact Tamura introduces all definitions for semigroups, but then mentions that they can be copied for arbitrary varieties of algebras). The fact that (iii) is equivalent to condition (ii) of Theorem \ref{thrm:AA.1} immediately follows from Tamura's definitions.
 \end{proof}

Following Tamura, Mal'tsev, and many other authors of the next generation, one would call $\X$ \emph{attainable} if all objects of $\C$ are $\X$-attainable. Actually we have
\begin{equation}\label{AA.5}
\text{semisimple} \;\Leftrightarrow \; \text{attainable}\; \Leftrightarrow\; \text{admissible}\; \Leftrightarrow\; \text{semi-left-exact}\; \Leftrightarrow\; \text{fibered}
\end{equation}
in the sense of:

\begin{theorem}\label{thrm:AA.6}
If $\C$ is a variety of $\Omega$-groups and $\X$ a subvariety of $\C$, then the following conditions are equivalent:
\begin{enumerate}[\em (i)]
\item $\X$ is a semisimple class in $\C$ in the sense of Kurosh-Amitsur radical theory;
\item $\X$ is an attainable subvariety of $\C$;
\item every object $A$ in $\C$ is $F$-admissible;
\item the reflection $F : \C\to\X$ is semi-left-exact;
\item the reflection $F : \C\to\X$ is (essentially) a fibration.
\end{enumerate}
\end{theorem}

\begin{proof}
(i) $\Leftrightarrow$ (ii) is `folklore', but since we could not find it explicitly mentioned in the literature, let us explain:

According to (4.2) of Theorem 4 in \cite{M1980} of R. Mlitz, (i) means that the following conditions hold:
\begin{enumerate}
\item[(i1)] $\X$ is closed in $\C$ under subdirect products;
\item[(i2)] if $X$ is in $\X$ and $Y = \Ker(f)\neq 0$ for some morphism $f$ with domain $X$, then $Y$ has a non-zero quotient in $\X$;
\item[(i3)] every $A$ in $\C$ satisfies condition (iii) of Theorem \ref{thrm:AA.1}.
\end{enumerate}
Since conditions (i1) and (i2) trivially follow from the fact that $\X$ is a subvariety of $\C$, this proves (i) $\Leftrightarrow$ (ii). For (ii) $\Leftrightarrow$ (iii) and (ii) $\Leftrightarrow$ (iv) see Remark \ref{rem:AA.2}. \end{proof}

\begin{remark}
 \label{rem:AA.7}
\begin{enumerate}
\item The algebraic versions of our arguments used in the proofs of the last assertion of Theorem \ref{thrm:AA.1} and of Theorem \ref{thrm:AA.3} are known in more general contexts of universal algebra and radical theory. The known context for attainability implying extension closedness is especially general, and the result goes back to the above-mentioned paper \cite{M1967}. On the other hand, our proof of Theorem \ref{thrm:AA.3} restricted to $\Omega$-groups is nothing but a simplified version of arguments used in \cite{M1980}. It would be interesting to remove the assumption on normality of the monomorphism $K(K(A))\to A$ in that theorem. It indeed can be removed when $\C$ is the variety of (associative) rings, as shown by R. Wiegandt \cite{W1974} and in some other cases (see Section 3.20 of B. J. Gardner and R. Wiegandt \cite{GW2004} and references at the end of it). But it cannot be removed even in some closely related categories (see e.g. B. J. Gardner \cite{G1980} and S. Veldsman \cite{V1988}). In the case of groups this is trivial since there are no extension closed (proper) subvarieties of the variety of groups, as observed in \cite{M1967} referring to results of B. H., H., and P. M. Neumann \cite{NNN1962}.

\item In the case of (associative non-unital) rings, we obtain an interesting conclusion: the (quite non-trivial!) complete description of extension closed varieties of rings, due to B. J. Gardner and P. N. Stewart \cite{GS1975}, also repeated in the book \cite{GW2004}, can also be considered as the complete description of semi-left-exact reflections from the variety of rings to its subvarieties.
\end{enumerate}
\end{remark}

Applying a part of Theorem 42 of \cite{BCGS08} to the special situation where $\C$ is a variety of $\Omega$-groups, we obtain the following result:
\begin{proposition}\label{thrm:AA.8}
If $\C$ is a variety of $\Omega$-groups and $\X$ a subvariety of $\C$, then the following conditions are equivalent:
\begin{enumerate}[\em (i)]
\item $\X$ is a semisimple class in $\C$ in the sense of Kurosh-Amitsur radical theory whose corresponding radical class is hereditary;
\item the reflection $F : \C\to\X$ is a (regular epi)-reflective protolocalization.
\end{enumerate}
\end{proposition}

\begin{remark}
\label{rem:AA.9}
For a reflection $F : \C\to\X$ of a variety $\C$ of $\Omega$-groups to a subvariety $\X$, in addition to the equivalences \eqref{AA.5}, we have now:
\begin{equation}\label{AA.6}
 \text{protolocalization } \Leftrightarrow \text{ semisimple with hered. radical } \Rightarrow  \text{ semisimple } \Rightarrow \text{ extension closed,}
\end{equation}
 with known counter-examples to the converse of the last implication (which we have according to Remark \ref{rem:AA.7}). For a counter-example of the previous implication, consider the Burnside variety ${\mathbb B}_6$ of groups of exponent $6$, i.e. satisfying the identity $x^6=1$ and its subvariety ${\mathbb B}_3$ determined by the identity $x^3=1$. ${\mathbb B}_3$ is a torsion-free subvariety of ${\mathbb B}_6$, whose radical is not hereditary (this follows from the results in \cite{GJ}, see in particular its Example $5.1$).
\end{remark}
\begin{remark}
\label{AA.10}
 \begin{enumerate}
 \item We do not know any example of a protolocalization that is not a localization in the cases where $\C$ is either the category of rings or the category of commutative rings (according to the results of \cite{GS1975}, these cases are obviously of interest). However, there are many examples of protoadditive reflections in these categories: for instance, such is the reflector from the category of commutative rings to its torsion-free subcategory of \emph{reduced} rings (i.e. those rings satisfying the implications $x^n=0 \Rightarrow x=0$, for any $n\ge 1$).
 \item\label{AA.10.1}  Let us also point out that there are no non-trivial localizations, neither of the category of rings nor of the category of commutative rings, that are reflections to subvarieties. For, let $F : \C\to\X$ be the canonical reflection to a subvariety with $\C$ being the category of commutative rings. If $A$ is the free commutative ring on a set $S$, then it is a subring of the polynomial ring $\bkZ[S]$ and therefore it is also a subring of the rational function field $\bkQ(S)$. Consider $F(\bkQ(S))$. If $\bkQ(S)$ belongs to $\X$ for every $S$, then $\X$, being a subvariety of $\C$, contains all free commutative rings, and therefore contains all commutative rings. If $\bkQ(S)$ does not belong to $\X$, then $F(\bkQ(S)) = 0$ since $F(\bkQ(S))$ must be a quotient ring of the field $\bkQ(S)$. But, for a non-empty $S$, this means that $F$ carries the monomorphism $A\to\bkQ(S)$ to a non-monomorphism $F(A)\to 0$ (since $F(A)$ is nothing but the free algebra in $\X$ on $S$). The same arguments can be used for non-commutative rings since every free ring can be embedded into a skew field (=division ring), as shown by P. M. Cohn \cite{C1961}.

\item\label{AA.10.3} A further simplified version of the same argument negatively answers an open question of \cite{BCGS08}, namely the question whether the reflection of the variety of commutative von Neumann regular rings to the variety of Boolean rings is a localisation. For, consider the embedding $\bkF_2\to\bkF_4$ of the two-elements field to the four-elements field. Since $\bkF_2$ is Boolean and $\bkF_4$ is not, we have $F(\bkF_2)=\bkF_2$ and $F(\bkF_4) = 0$; hence $F$ carries a monomorphism $\bkF_2\to\bkF_4$ to a non-monomorphism.
\end{enumerate}
\end{remark}

\end{document}